\documentclass[12pt]{amsart}
\usepackage{amsfonts}
\usepackage{scalerel}       
\usepackage{txfonts}
\usepackage{amssymb}
\usepackage{eucal}
\usepackage{graphicx}
\usepackage{amsmath}
\usepackage{amscd}
\usepackage[all]{xy}           
\usepackage{tikz}
\usepackage{amsfonts,latexsym}
\usepackage{xspace}
\usepackage{epsfig}
\usepackage{float}
\usepackage{color}
\usepackage{fancybox}
\usepackage{colordvi}
\usepackage{multicol}
\usepackage{colordvi}
\usepackage[colorlinks,final,backref=page,hyperindex,hypertex]{hyperref}
\usepackage[active]{srcltx} 

\usepackage{ifpdf}
\ifpdf
  \usepackage[colorlinks,final,backref=page,hyperindex]{hyperref}
\else
  \usepackage[colorlinks,final,backref=page,hyperindex,hypertex]{hyperref}
\fi

\newtheorem{theorem}{Theorem}[section]
\newtheorem{prop}[theorem]{Proposition}
\theoremstyle{definition}
\newtheorem{defn}[theorem]{Definition}
\newtheorem{lemma}[theorem]{Lemma}
\newtheorem{coro}[theorem]{Corollary}
\newtheorem{prop-def}{Proposition-Definition}[section]
\newtheorem{coro-def}{Corollary-Definition}[section]

\newtheorem{remark}[theorem]{Remark}

\newtheorem{exam}{Example}[section]

\newtheorem{problem}[theorem]{Problem}

\newcommand{\nc}{\newcommand}
\newcommand{\delete}[1]{}

\nc{\tred}[1]{\textcolor{red}{#1}} \nc{\tblue}[1]{\textcolor{blue}{#1}} \nc{\tgreen}[1]{\textcolor{green}{#1}} \nc{\tpurple}[1]{\textcolor{purple}{#1}} \nc{\btred}[1]{\textcolor{red}{\bf #1}} \nc{\btblue}[1]{\textcolor{blue}{\bf #1}} \nc{\btgreen}[1]{\textcolor{green}{\bf #1}} \nc{\btpurple}[1]{\textcolor{purple}{\bf #1}}

\renewcommand{\Bbb}{\mathbb}

\newcommand{\efootnote}[1]{}
\renewcommand{\textbf}[1]{}

\nc{\mlabel}[1]{\label{#1}}  
\nc{\mcite}[2][]{\cite[#1]{#2}}  
\nc{\mref}[1]{\ref{#1}}  
\nc{\mbibitem}[1]{\bibitem{#1}} 

\delete{
\nc{\mlabel}[1]{\label{#1}  
{\hfill \hspace{1cm}{\bf{{\ }\hfill(#1)}}}}
\nc{\mcite}[2][1]{\cite[#1]{#2}{{\bf{{\ }(#2; #1)}}}}  
\nc{\mref}[1]{\ref{#1}{{\bf{{\ }(#1)}}}}  
\nc{\mbibitem}[1]{\bibitem[\bf #1]{#1}} 
}

\renewcommand\geq{\geqslant}
\renewcommand\leq{\leqslant}

\renewcommand\bar[1]{\overline{#1}}

\newcommand\reallywidehat[1]{\arraycolsep=0pt\relax%
\begin{array}{c}
\stretchto{
  \scaleto{
    \scalerel*[\widthof{\ensuremath{#1}}]{\kern-.5pt\bigwedge\kern-.5pt}
    {\rule[-\textheight/2]{1ex}{\textheight}} 
  }{\textheight} %
}{0.5ex}\\           
#1\\                 
\rule{-1ex}{0ex}
\end{array}
}

\nc{\into}{I}
\nc{\rbw}{\mathfrak{R}} \nc{\brp}{\mathrm{brp}} \nc{\lead}{\mathrm{Lead}} \nc{\Id}{\mathrm{Id}} \nc{\Irr}{\mathrm{Irr}}
\nc{\vx}{\sigma} \nc{\vy}{\tau} \nc{\dvx}{\sigma^{(1)}} \nc{\dvy}{\tau^{(1)}} \nc{\done}{\vep} \nc{\mcitep}[1]{\mcite{#1}} \nc{\wt}{\mathrm{wt}} \nc{\bre}[1]{|#1|} \nc{\mapmonoid}{\frakM} \nc{\disjoint}{\frakM'}
\nc{\ncpoly}[1]{\langle #1\rangle}  
\nc{\mapm}[1]{\lfloor\!|{#1}|\!\rfloor}
\nc{\diff}[1]{{}^\NC\{ #1 \}} \nc{\disj}[1]{\{{#1}\}'} \nc{\mdisj}[1]{\frakM'(#1)} \nc{\brho}{\bar{\rho}} \nc{\om}{\bar{\frakm}} \nc{\frakn}{\mathfrak n} \nc{\ddeg}[1]{^{(#1)}} \nc{\opset}{X} \nc{\genset}{{Z}} \nc{\NC}{\mathrm{{NC}}} \nc{\leaf}{\mathrm{leaf}} \nc{\twig}{\mathrm{twig}} \nc{\fe}{\mathrm{fl}} \nc{\munderline}[1]{#1} \nc{\bo}{o} \nc{\dep}{\mathrm{depth}} \nc{\ofe}{\mathrm{ofl}} \nc{\dfe}{\mathrm{dfe}} \nc{\fex}{\mathrm{fex}} \nc{\dl}{\mathrm{dlex}} \nc{\db}{\mathrm{db}} \nc{\lex}{\mathrm{lex}} \nc{\clex}{\mathrm{clex}} \nc{\dgp}{\mathrm{dgp}} \nc{\dgx}{\mathrm{dgx}} \nc{\br}{\mathrm{br}} \nc{\obd}{\mathrm{odb}} \nc{\ob}{\mathrm{ob}}
\nc{\pie}{\mathrm{PIE}}
\nc{\rbo}{\mathrm{RBO}}
\nc{\supp}{\mathcal{S}}
\nc{\nul}{\mathcal{Z}}

\nc{\bin}[2]{ (_{\stackrel{\scs{#1}}{\scs{#2}}})}  
\nc{\binc}[2]{ \left (\!\! \begin{array}{c} \scs{#1}\\
    \scs{#2} \end{array}\!\! \right )}  
\nc{\bincc}[2]{  \left ( {\scs{#1} \atop \vspace{-1cm}\scs{#2}} \right )}  
\nc{\bs}{\bar{S}} \nc{\cosum}{\sqsubset} \nc{\la}{\longrightarrow} \nc{\rar}{\rightarrow} \nc{\dar}{\downarrow} \nc{\dprod}{**} \nc{\dap}[1]{\downarrow \rlap{$\scriptstyle{#1}$}} \nc{\md}[1]{\bar{#1}} \nc{\uap}[1]{\uparrow \rlap{$\scriptstyle{#1}$}} \nc{\defeq}{\stackrel{\rm def}{=}} \nc{\disp}[1]{\displaystyle{#1}} \nc{\dotcup}{\ \displaystyle{\bigcup^\bullet}\ } \nc{\gzeta}{\bar{\zeta}} \nc{\hcm}{\ \hat{,}\ } \nc{\hts}{\hat{\otimes}} \nc{\barot}{{\otimes}} \nc{\free}[1]{\bar{#1}} \nc{\uni}[1]{\tilde{#1}} \nc{\hcirc}{\hat{\circ}} \nc{\leng}{\ell} \nc{\lleft}{[} \nc{\lright}{]}
\nc{\lb}{[} 
\nc{\rb}{]} 
\nc{\curlyl}{\left \{ \begin{array}{c} {} \\ {} \end{array}
    \right.  \!\!\!\!\!\!\!}
\nc{\curlyr}{ \!\!\!\!\!\!\!
    \left. \begin{array}{c} {} \\ {} \end{array}
    \right \} }
\nc{\longmid}{\left | \begin{array}{c} {} \\ {} \end{array}
    \right. \!\!\!\!\!\!\!}
\nc{\onetree}{\bullet} \nc{\ora}[1]{\stackrel{#1}{\rar}}
\nc{\ola}[1]{\stackrel{#1}{\la}}
\nc{\ot}{\otimes} \nc{\mot}{{{\boxtimes\,}}} \nc{\otm}{\overline{\boxtimes}} \nc{\sprod}{\bullet} \nc{\scs}[1]{\scriptstyle{#1}} \nc{\mrm}[1]{{\rm #1}} \nc{\msum}{\sum\limits}
\nc{\margin}[1]{\marginpar{\rm #1}}   
\nc{\dirlim}{\displaystyle{\lim_{\longrightarrow}}\,} \nc{\invlim}{\displaystyle{\lim_{\longleftarrow}}\,} \nc{\mvp}{\vspace{0.3cm}} \nc{\tk}{^{(k)}} \nc{\tp}{^\prime} \nc{\ttp}{^{\prime\prime}} \nc{\svp}{\vspace{2cm}} \nc{\vp}{\vspace{8cm}} \nc{\proofbegin}{\noindent{\bf Proof: }}
\nc{\proofend}{$\blacksquare$ \vspace{0.3cm}}
\nc{\modg}[1]{\!<\!\!{#1}\!\!>}
\nc{\intg}[1]{F_C(#1)} \nc{\lmodg}{\!<\!\!} \nc{\rmodg}{\!\!>\!} \nc{\cpi}{\widehat{\Pi}}
\nc{\sha}{{\mbox{\cyr X}}}  
\nc{\shap}{{\mbox{\cyrs X}}} 
\nc{\shpr}{\diamond}    
\nc{\shp}{\ast} \nc{\shplus}{\shpr^+}
\nc{\shprc}{\shpr_c}    
\nc{\msh}{\ast} \nc{\zprod}{m_0} \nc{\oprod}{m_1} \nc{\vep}{\varepsilon} \nc{\labs}{\mid\!} \nc{\rabs}{\!\mid}
\nc{\astarrow}{\overset{\raisebox{-3pt}{$\ast$}}{\rightarrow}}

\nc{\dth}{d} \nc{\mmbox}[1]{\mbox{\ #1\ }} \nc{\fp}{\mrm{FP}} \nc{\rchar}{\mrm{char}} \nc{\Fil}{\mrm{Fil}} \nc{\Mor}{Mor\xspace} \nc{\gmzvs}{gMZV\xspace} \nc{\gmzv}{gMZV\xspace} \nc{\mzv}{MZV\xspace} \nc{\mzvs}{MZVs\xspace} \nc{\Hom}{\mrm{Hom}} \nc{\id}{\mrm{id}} \nc{\im}{\mrm{im}} \nc{\incl}{\mrm{incl}} \nc{\map}{\mrm{Map}} \nc{\mchar}{\rm char} \nc{\nz}{\rm NZ}

\nc{\Alg}{\mathbf{Alg}} \nc{\Bax}{\mathbf{Bax}} \nc{\bff}{\mathbf f} \nc{\bfk}{{\bf k}} \nc{\bfone}{{\bf 1}} \nc{\bfx}{\mathbf x} \nc{\bfy}{\mathbf y}
\nc{\base}[1]{\bfone^{\otimes ({#1}+1)}} 
\nc{\Cat}{\mathbf{Cat}} \delete{}
\nc{\detail}{\marginpar{\bf More detail}
    \noindent{\bf Need more detail!}
    \svp}
\nc{\Int}{\mathbf{Int}} \nc{\Mon}{\mathbf{Mon}}
\nc{\rbtm}{{shuffle }} \nc{\rbto}{{Rota-Baxter }} \nc{\remarks}{\noindent{\bf Remarks: }} \nc{\Rings}{\mathbf{Rings}} \nc{\Sets}{\mathbf{Sets}}

\nc{\BA}{{\Bbb A}} \nc{\CC}{{\Bbb C}} \nc{\DD}{{\Bbb D}} \nc{\EE}{{\Bbb E}} \nc{\FF}{{\Bbb F}} \nc{\GG}{{\Bbb G}} \nc{\HH}{{\Bbb H}} \nc{\LL}{{\Bbb L}} \nc{\NN}{{\Bbb N}} \nc{\KK}{{\Bbb K}} \nc{\QQ}{{\Bbb Q}} \nc{\RR}{{\Bbb R}} \nc{\TT}{{\Bbb T}} \nc{\VV}{{\Bbb V}} \nc{\ZZ}{{\Bbb Z}}

\nc{\cala}{{\mathcal A}} \nc{\calc}{{\mathcal C}} \nc{\cald}{{\mathcal D}} \nc{\cale}{{\mathcal E}} \nc{\calf}{{\mathcal F}} \nc{\calg}{{\mathcal G}} \nc{\calh}{{\mathcal H}} \nc{\cali}{{\mathcal I}} \nc{\call}{{\mathcal L}} \nc{\calm}{{\mathcal M}} \nc{\caln}{{\mathcal N}} \nc{\calo}{{\mathcal O}} \nc{\calp}{{\mathcal P}} \nc{\calr}{{\mathcal R}} \nc{\cals}{{\mathcal S}} \nc{\calt}{{\mathcal T}} \nc{\calw}{{\mathcal W}} \nc{\calk}{{\mathcal K}} \nc{\calx}{{\mathcal X}}
\nc{\calz}{{\mathcal Z}}
 \nc{\CA}{\mathcal{A}}

\nc{\fraka}{{\mathfrak a}} \nc{\frakA}{{\mathfrak A}} \nc{\frakb}{{\mathfrak b}} \nc{\frakB}{{\mathfrak B}} \nc{\frakD}{{\mathfrak D}} \nc{\frakH}{{\mathfrak H}} \nc{\frakM}{{\mathfrak M}} \nc{\bfrakM}{\overline{\frakM}} \nc{\frakm}{{\mathfrak m}} \nc{\frakP}{{\mathfrak P}} \nc{\frakN}{{\mathfrak N}} \nc{\frakp}{{\mathfrak p}} \nc{\frakS}{{\mathfrak S}} \nc{\frakx}{{\mathfrak x}} \nc{\ox}{\bar{\frakx}} \nc{\frakX}{{\mathfrak X}} \nc{\fraky}{{\mathfrak y}} \nc\dop{\delta}
\nc{\Reduce}{{\rm Red}}

\font\cyr=wncyr10 \font\cyrs=wncyr7
\nc{\redt}[1]{\textcolor{red}{#1}}

\nc{\lge}{{\langle}}
\nc{\rge}{{\rangle}}
\nc{\intl}{\mathrm{int}}
\nc{\unl}{\mathrm{unl}}
\nc{\hs}{\text{Hurwitz series}}
\nc{\ord}{\mathrm{ord}}
\nc{\lc}{\lfloor} \nc{\rc}{\rfloor}
\nc{\pal}{\partial}

\nc{\nonz}[1]{#1^\times}
\nc{\NNP}{\nonz{\NN}}
\nc{\stdint}[1]{J_{#1}}
\nc{\dualmod}[1]{#1^*}
\nc{\alghom}[1]{#1^\bullet}
\nc{\End}{\mathrm{End}}
\nc{\rng}{\mathrm{\mathcal{R}}}
\nc{\codim}{\mathrm{codim}}
\nc{\evl}{\mathrm{ev}}
\nc{\bx}{\ast}

\begin{document}
\title[Interlacing, integrals of Hurwitz series and differential equations]{Interlacing, integrals of Hurwitz series and differential equations}


\author{Shanghua Zheng}
\address{Department of Mathematics, Jiangxi Normal University, Nanchang, Jiangxi 330022, China}
\email{zhengsh@jxnu.edu.cn}

\hyphenpenalty=8000
\date{\today}

\begin{abstract}
In this paper, we first introduce the unlacing of Hurwitz series,  which can be viewed as an inverse of interlacing, and develop the basic properties of unlacing, interlacing and integral of Hurwitz series.   We then show that the multiplication of interlacing of Hurwitz series by basis elements can be also rewritten as an interlacing, and provide the explicit multiplication formula in terms of integrals and Hadamard products. Finally, we investigate the natural exponential function $\exp$, and realize the method of  reduction of order in the ring of Hurwitz series by using $\exp$.
\end{abstract}

\keywords{Hurwitz series, interlacing, integral, differential equation\\
\qquad 2020 Mathematics Subject Classification. 12H05, 16W60, 47G10, 12H20}


\maketitle

\tableofcontents

\hyphenpenalty=8000 \setcounter{section}{0}


\section{Introduction}\mlabel{sec:int}
\smallskip
The notion of interlacing, as a basic approach to obtain new algebraic structures from old, is widely employed in algebra, such as interlacing modules~\cite{DK}  and interlacing polynomials~\cite{Joh}. In particular, the interlacing of linear recursive sequences, was proposed by Larson and  Taft~\cite{LT}, for describing the algebraic structure of linear recursive sequences with respect to  Hadamard products. Especially, by applying the results on the Hopf algebra structure of linearly recursive sequences, the authors showed that a linear recursive sequence $f$ is Hadamard invertible if and only if $f$ is everywhere nonzero, and except for a finite number of terms, is the interlacing of geometric series. See~\cite{PT} for further details concerning the structure of Hopf algebras of linear recursive sequences.

The ring of Hurwitz series, as a variant of formal power series ring, was proposed for studying the differential algebra~\cite{K1}. Lately, the concept of a $\lambda$-differential algebra (also called a differential algebra of weight $\lambda$), as a common generalization of differential algebra (when $\lambda=0$) and difference algebra (when $\lambda=1$), was introduced in ~\cite{GuK}. In order to develop the cofree $\lambda$-differential algebra on any $\bfk$-algebra $A$, the notion of a $\lambda$-Hurwitz series over $A$ was also presented in ~\cite{GuK}. Furthermore, the $\lambda$-Hurwitz series  carries  a $\lambda$-differential Rota-Baxter algebra structure. See~\cite{Gub,YGT,ZGGS} for more details on Rota-Baxter algebras.

As it turns out,  the algebraic structures and properties of  Hurwitz series ring over  an algebra $A$ of positive characteristic  are closely related to those of $A$. Motivated by this, the Hurwitz series ring has also attracted quite much attention recently. For instance, in ~\cite{Be}, Benhissi depicted
 the Hurwitz series ring over $A$ has PF-properties if and only if $A$ is PF-ring and a torsion free module provided a extra condition.
In addition, according to  McCoy's Theorem on a commutative ring,  McCoy and CN-properties of Hurwitz series ring are also considered in~\cite{NRA}. What is more, the ring of Hurwitz series over a noncommutative ring was introduced by Paykan~\cite{Pa}, and it is proved that the skew Hurwitz series ring $(HA,\alpha)$, where $\alpha:A\to A$ is a ring homomorphism, is a 2-primal ring if and only if $A$ is 2-primal ring under some appropriate conditions.

More recently, from a differential equation viewpoint, the authors ~\cite{GK1} developed the interlacing of Hurwitz series, and proved that solutions of $n$th-order monic linear homogeneous differential equations with constant coefficients  can be considered as an interlacing of solutions of a first-order system of differential equations.

To make our point more precise, let us briefly recall the main result of~\cite{GK1}.
\vspace{0.1cm}

\begin{theorem} ~\cite[Theorem~3.4]{GK1}
Let $A$ be a commutative ring with identity element. Let $HA$ be the ring of Hurwitz series.
Let $z=(z_0,z_1,\cdots,z_{n-1})\in HA^n$ and let $u=\intl (z_0,z_1,\cdots,z_{n-1})$ be the interlacing of $z_0,z_1,\cdots,z_{n-1}$.  Then $u$ is a solution of $n$th-order monic linear homogeneous differential equations
$$L(y)=\pal^ny+\sum_{i=0}^{n-1}a_i\pal^i y=0, \quad a_i\in A$$
if and only if  $Z=z^t$, that is, the transpose of $z$, is a solution of the $n\times n$ matrix equation
$$LZ'=UZ,$$
where \,$Z'=(\pal(z_0),\pal(z_1),\cdots,\pal(z_{n-1}))^t$,
$$L=\left(
      \begin{array}{cccccc}
        1 & 0 & 0 &\cdots & 0& 0 \\
        a_{n-1} & 1 & 0 & \cdots & 0 & 0 \\
        a_{n-2} & a_{n-1} &1 & \cdots & 0 & 0 \\
        \vdots & \vdots & \vdots & \ddots & \vdots & \vdots \\
        a_2 & a_3 & a_4 & \cdots & 1 & 0 \\
        a_1 & a_2 & a_3 & \cdots & a_{n-1} & 1 \\
      \end{array}
    \right)
\,\,\text{and}\,\,\,
U=\left(
      \begin{array}{cccccc}
        -a_0 & -a_1 & -a_2 &\cdots & -a_{n-2}& -a_{n-1} \\
        0 & -a_0 & -a_1 & \cdots & -a_{n-3} & -a_{n-2} \\
        0 & 0&-a_0& \cdots & -a_{n-4} & -a_{n-3} \\
        \vdots & \vdots & \vdots & \ddots & \vdots & \vdots \\
        0 & 0 & 0 & \cdots & -a_0 & -a_1 \\
        0 & 0 & 0 & \cdots &0 & -a_0 \\
      \end{array}
    \right).
$$
\mlabel{thm:gkthm}
\end{theorem}
As is well known, solving  differential equations with variable coefficients  is more difficult than differential equations with  constant coefficients. Likewise, it is a challengeable problem to generalize Theorem ~\mref{thm:gkthm} to the more general case that the coefficients of equations are arbitrary variable, such as
$$L(y)=\pal^ny+\sum_{i=0}^{n-1}f_i\pal^i y=0, \quad f_i\in HA.$$
One of key to the above question is that the product of a variable and an interlacing should be rewritten as a new interlacing. For this,  the authors proposed the following  open problem in~\cite{GK1}.
\begin{problem}
Whether or not we can find a more general formula in parallel with the formula
$$a\, (\intl(z_0,z_1,\cdots,z_{n-1}))=\intl(az_0,az_1,\cdots,az_{n-1}),$$
where $a\in A$ is a constant.
\end{problem}
In this paper, we give a definite answer to the above question.
In order to achieve this goal, the first step is to provide a multiplication formula for basis elements $\lge m\rge $ and $\intl(z_0,z_1,\cdots,z_{n-1})$, and then we obtain  a multiplication formula for  arbitrary Hurwitz series $f\in HA $ and $\intl(z_0,z_1,\cdots,z_{n-1})$ by linearity.
Meanwhile, as an application of interlacing to differential equations, we list some  examples, such as $\exp(\sin)$,  the composition of  the natural exponential function $\exp$ and the trigonometric function $\sin$. These examples will result in several new sequences in OEIS~\cite{Slo}.

After that, applying  the natural exponential function $\exp$, we realize  the general solutions of a second-order linear homogeneous differential equation. Finally, we investigate the method of  reduction of order for solving  differential equations in the ring of Hurwitz series.

\smallskip

\noindent
{\bf Notations.}
Unless otherwise specified, all rings are commutative with identity element.  We denote by $\NN$ the natural numbers, $\NN^+$ the positive integers, $\QQ$ the rational numbers, and $\RR$ the real numbers. Fix a positive integer $n\in\NN^+$, $\NN_n=\{0,1,\cdots,n-1\}$.

\section{Hurwitz series}
\mlabel{sec:prodHA}

In this section, we first recall some basic notations and  results of Hurwitz series. See ~\cite{GK1, K1} for a more detailed treatment.
Let $A$ be a commutative ring with identity element $1$. Denote the set $A^{\NN}:=\{f:\NN\to A\}$ consisting of infinite sequences in $A$, which is called the ring of Hurwitz series over $A$. We denote by $HA$ the ring of Hurwitz series. The elements of  $HA$ will always be written in the form $(z(n)):=(z(0),z(1),z(2),\cdots,z(n),\cdots)$, where $z(i)\in A, i\geq 0$.
Furthermore, addition and multiplication in $HA$ are defined by:
\begin{eqnarray}
&&(a(n))+(b(n)):=(c(n)),\,\quad \text{where}\quad c(n)=a(n)+b(n)\nonumber\\
&&(a(n))\cdot(b(n)):=(c(n)),\quad\text{where}\quad c(n)=\sum_{i=0}^n{n\choose i} a(i)b(n-i),\,\, n\in\NN.
\mlabel{eq:Hproduct}
\end{eqnarray}
Then $0:=(0,0,0,\cdots,)$ is a  zero element under addition and $1:=(1,0,0,\cdots)$ is the identity element of $HA$ under product.
Let $\bfk$ be a unitary  commutative ring. Let $A$ be a commutative $\bfk$-algebra.  Recall from~\cite{GuK} that $HA$ is a differential $\bfk$-algebra with  derivation  defined by
$$\partial_A:HA\to HA,\quad (a(0),a(1),a(2),\cdots)\mapsto (a(1),a(2),a(3),\cdots).$$
Further,  $HA$ is also a Rota-Baxter $\bfk$-algebra with  a Rota-Baxter operator of weight $0$ (that is,  integral operator)  given by
$$\int:HA\to HA,\quad (a(0),a(1),a(2),\cdots)\mapsto (0,a(0),a(1),a(2),\cdots).$$
Note that $\partial_A \circ \int =\id_{HA}$. Thus $(HA,\partial_A, \int)$ is a differential Rota-Baxter $\bfk$-algebra of weight $0$.
We see that there exists a natural ring homomorphism
$$\vep: HA\to A,\,h\mapsto h(0).$$
For any natural number $m$, denote
\begin{equation}
\langle m\rangle :=(\underbrace{0,\cdots 0}_{m},1,0,\cdots)=(\delta_{i,m})_{i\geq0}\in HA.
\end{equation}
For instance,
\begin{equation}
\lge 0\rge=(1,0,0,0,\cdots),\lge 1\rge=(0,1,0,0,\cdots),\lge 2\rge=(0,0,1,0,\cdots).
\end{equation}
We denote $u:=\lge 1\rge$.
Note that, for every $z=(z(0),z(1),z(2),\cdots)\in HA$, we have
\begin{equation}
z=\sum_{m=0}^\infty z(m)\lge m\rge.
\mlabel{eq:lincom}
\end{equation}
This means that $z$ can be expressed as a unique linear combination of $\lge m\rge$.
Thus, the set $\{\lge m\rge\,|\, m\in\NN\}$ forms a basis of Hurwtiz series.

\subsection{Interlacing and integrals of Hurwitz series}
\mlabel{subsec:IntandIng}

For all $f\in HA$ and all $n\in\NN$, we denote by
$$\int^n f:=\underbrace{\int\cdots\int}_n f$$
the $n$-fold integral of $f$ with the convention that $\int^0=\id$.
\begin{prop}\mlabel{prop:basismn}
Let $m,n\in\NN$ and let $k\in\NN^+$. Then
\begin{enumerate}
\item \mlabel{it:parba}
$$\partial_A(\lge m+1\rge)=\lge m\rge\quad\text{and}\quad\partial_A^m(\lge m\rge)=\lge 0\rge.$$
\item\mlabel{it:intba}
\begin{equation}
\int^n\lge m\rge=\lge m+n\rge,\quad\text{and further}\quad\lge n\rge=\int^n\lge 0\rge.
\mlabel{eq:intba}
\end{equation}
\item\mlabel{it:intpar}
$$(\int \circ \partial_A) (\lge k\rge)=\lge k\rge.$$
\item\mlabel{it:proba}
\begin{equation}
\lge m\rge \lge n\rge={m+n\choose m}\lge m+n\rge.
\mlabel{it:mnpro}
\end{equation}
\end{enumerate}

\end{prop}
\begin{proof}
We prove only Item~(\mref{it:proba}) by induction on $m+n\geq 0$. Items~(\mref{it:parba}), (\mref{it:intba}) and (\mref{it:intpar}) follow from the definition of $\partial_A$ and $\int$.  The case when $m=0$ or $n=0$ is trivial. Assume that $m\geq 1$ and $n\geq 1$. Using the formula for integration by part and Eq.~(\mref{eq:intba}), we get
\begin{eqnarray*}
\lge m\rge \lge n\rge&=& \Big(\int \lge m-1\rge \Big)\Big(\int \lge n-1\rge\Big)\\
&=&\int\bigg(\lge m-1\rge\int(\lge n-1\rge )+\int(\lge m-1\rge)\lge n-1\rge\bigg)\\
&=&\int\bigg(\lge m-1\rge \lge n\rge+\lge m\rge \lge n-1\rge\bigg)\\
&=&\int\bigg({m-1+n\choose n}\lge m+n-1\rge+{m+n-1\choose n-1}\lge m+n-1\rge\bigg)\\
&=&\int {m+n\choose n}\lge m+n-1\rge\\
&=&{m+n\choose n}\lge m+n\rge.
\end{eqnarray*}
This completes the proof of Item~(\mref{it:proba}).
\end{proof}

\begin{defn} ~\cite[Definition 3.1]{GK1}\mlabel{defn:intl1}
Let $n\in \NN^+$ be fixed. Let $z_0,z_1,\cdots,z_{n-1}\in HA$. For every $k\in\NN$,  define the following finite sequences having $n$ elements in $A$
\begin{equation}
s_{n}(k):=\{z_0(k),z_1(k),\cdots,z_{n-1}(k)\}.
\end{equation}
The {\bf interlacing } of  $z_0,z_1,\cdots,z_{n-1}$ is the Hurwitz series consisting of all $s_{n}(k)$ with $k\geq 0$
\begin{equation}
\intl(z_0,z_1,\cdots,z_{n-1}):=(s_{n}(0),s_{n}(1),\cdots,s_{n}(k),\cdots),
\mlabel{eq:intl1}
\end{equation}
that is,
\begin{eqnarray}
&&\intl(z_0,z_1,\cdots,z_{n-1})\mlabel{eq:intl2}\\
&=&(z_0(0),z_1(0),\cdots,z_{n-1}(0),z_0(1),z_1(1),\cdots,z_{n-1}(1),\cdots,z_0(k),z_1(k),\cdots,z_{n-1}(k),\cdots).\nonumber
\end{eqnarray}
\end{defn}
\begin{exam}
Take $n=3$ in Eq.~(\mref{eq:intl2}). Let $z_0=(0,0,0,\cdots), z_1=(1,1,1,\cdots),z_2=(2,2,2,\cdots)$. Then $\intl(z_0,z_1,z_2)=(0,1,2,0,1,2,\cdots)$.
\end{exam}
The {\bf floor function}, also called the greatest integer function or integer value,  is the function that  gives the greatest integer less than or equal to $r\in\RR$, denoted by $\lc r\rc$.
Fix a positive integer $n\in\NN^+$. Let $\NN_n=\{0,1,\cdots,n-1\}$.
For any $m\in \NN$,
we set $$\widehat{m}=\lfloor \frac{m}{n}\rfloor \quad\text{and}\quad \overline{m}=m-\widehat{m}n.$$
Then
\begin{equation}
m=\widehat{m}n+\overline{m}.
\mlabel{eq:mdeco}
\end{equation}
and
\begin{equation}
\widehat{m+nq}=\widehat{m}+q\quad \text{and}\quad \overline{m+nq}=\overline{m}\quad\text{for all}\,\, q\in\ZZ.
\mlabel{eq:overhat}
\end{equation}
\begin{prop}\mlabel{prop:intl2}
Let $z_0,z_1,\cdots,z_{n-1}\in HA$ and let $z\in HA$.  Then $z=\intl(z_0,z_1,\cdots,z_{n-1})$ if and only if $z(m)=z_{\overline{m}}(\widehat{m})$ for all $m\geq 0$.
\end{prop}
\begin{proof}We shall prove that $\intl(z_0,z_1,\cdots,z_{n-1})(m)=z_{\overline{m}}(\widehat{m})$ for all $m\geq 0$. For any given $m\geq 0$, by the Division Algorithm,
$$ m=n\widehat{m}+\overline{m},\quad \text{where}\,\, 0\leq \overline{m}\leq n-1.$$
By Eq.~(\mref{eq:intl1}), we obtain
$$\intl(z_0,z_1,\cdots,z_{n-1})(n\widehat{m}+\overline{m})\in s_n(\widehat{m})\,\Big(=\{z_0(\widehat{m}),z_1(\widehat{m}),\cdots,z_{n-1}(\widehat{m})\}\Big).$$
By Eq.~(\mref{eq:intl2}) and the equation $\overline{m}=m-n\widehat{m}$,  we have
$$\hspace{4.5cm}\intl(z_0,z_1,\cdots,z_{n-1})(n\widehat{m}+\overline{m})=z_{\overline{m}}(\widehat{m}).\hspace{4.5cm}\mbox{\qedhere}$$
\end{proof}
Then by Eq.~(\mref{eq:lincom}), we get
\begin{coro}\mlabel{coro:intcom}
Let $n\in\NN^+$ be fixed. Let $z_0,z_1,\cdots,z_{n-1}\in HA$ and let $z=\intl(z_0,z_1,\cdots,z_{n-1})$. Then
$$\intl(z_0,z_1,\cdots,z_{n-1})=\sum_{m=0}^\infty z_{\overline{m}}(\widehat{m})\lge m\rge.$$
\end{coro}
Furthermore, by Proposition~\mref{prop:basismn}~(\mref{it:mnpro}), we have
\begin{coro}\mlabel{coro:kinton}
Let $n\in\NN^+$ be fixed. Let $z_0,z_1,\cdots,z_{n-1}\in HA$ and let $z=\intl(z_0,z_1,\cdots,z_{n-1})$. For all $k\in\NN$,
\begin{enumerate}
\item\mlabel{it:kintn}
\begin{equation*}\lge k\rge\,\intl(z_0,z_1,\cdots,z_{n-1})=\sum_{m=0}^\infty z_{\overline{m}}(\widehat{m}){k+m\choose k}\lge k+m\rge.\end{equation*}
\item\mlabel{it:kplus1}
\begin{equation*}\lge k\rge\,\intl(z_0,z_1,\cdots,z_{n-1})=\frac{1}{k!}\lge 1\rge^{k}\,\intl(z_0,z_1,\cdots,z_{n-1}).
\end{equation*}
\end{enumerate}
\end{coro}
Denote by $HA^n$ the direct product $\underbrace{HA\times \cdots\times HA}_n$. We know that $HA^n$ together with addition and multiplication defined termwise, and  scalar multiplication $\cdot: A\times HA^n\to HA^n$, $(a,(f_0,f_1,\cdots,f_{n-1}))\mapsto (af_0,af_1,\cdots,af_{n-1})$ is an $A$-algebra.
\begin{lemma}~\cite[Lemma 3.3]{GK1}\mlabel{lem:intlin}
Let $n\in \NN^+$. Let $u:=(u_0,u_1,\cdots,u_{n-1})\in HA^n$ and let $v:=(v_0,v_1,\cdots,v_{n-1})\in HA^n$. Then
\begin{enumerate}
\item
$\intl(u+v)=\intl(u)+\intl(v)$.
\mlabel{it:addhom}
\item For any $a\in A$,
$a\,\intl(u)=\intl(au)$.
\mlabel{it:klin}
\item
For any given $0\leq i\leq n$,
\begin{equation}
\partial^i(\intl(u_0,u_1,\cdots,u_{n-1}))=\intl(u_i,u_{i+1},\cdots,u_{n+i-1}),
\end{equation}
where $u_{n+\ell}=\partial(u_\ell)$ for $0\leq \ell\leq i-1$.
\mlabel{it:derivation}
\item
$$\int\intl(u_0,u_1,\cdots,u_{n-1})=\intl\Big(\int u_{n-1},u_0,\cdots,u_{n-2}\Big).$$
\mlabel{it:intlac}
\end{enumerate}
\end{lemma}
\delete{
\begin{coro}
For every $0\leq m\leq n$,
\begin{equation}
\underbrace{\int\cdots\int}_m\intl(u_0,u_1,\cdots,u_{n-1}))=\intl(\int u_{n-m},\int u_{n-m+1},\cdots,\int u_{n-1}, u_0,\cdots,u_{n-m-1}).
\mlabel{eq:integral}
\end{equation}
\end{coro}
\begin{proof}
Use induction on $m\geq 0$.  The basic case $m=0$ is trivial. Assume Eq.~(\mref{eq:integral}) has been proved for $m\leq \ell<n$.  Consider $m=\ell +1$.
\begin{eqnarray*}
&&\underbrace{\int\cdots\int}_{\ell+1}\intl(u_0,u_1,\cdots,u_{n-1}))\\
&=&\int\intl(\int u_{n-\ell},\int u_{n-\ell+1},\cdots,\int u_{n-1}, u_0,\cdots,u_{n-\ell-1})\quad(\text{by the induction hypothesis)})\\
&=&\intl(\int u_{n-\ell-1},\int u_{n-\ell},\cdots,\int u_{n-1}, u_0,\cdots,u_{n-\ell-2}).\quad(\text{by Lemma~\mref{lem:intlin}(\mref{it:intlac}))})
\end{eqnarray*}
\end{proof}
}
\vspace{-8mm}
\begin{prop}\mlabel{prop:intmfold}
For every $m\geq 0$, we have
\begin{equation}
\int^m\intl(u_0,u_1,\cdots,u_{n-1})
=\intl\Big(\int^{\widehat{m}+1}u_{n-\overline{m}},\cdots,\int^{\widehat{m}+1} u_{n-1},\int^{\widehat{m}} u_0,\cdots,\int^{\widehat{m}}u_{n-\overline{m}-1}\Big).
\mlabel{eq:intmfold}
\end{equation}
\end{prop}
\begin{proof}
Use induction on $m\geq 0$.  The case $m=0$ is trivial. Assume Eq.~(\mref{eq:intmfold}) has been proved for $m\leq \ell$ with $\ell\geq 0$.  Consider $m=\ell+1$. \begin{eqnarray*}
&&\int^{\ell +1}\intl(u_0,u_1,\cdots,u_{n-1})\\
&=&\int\bigg(\int^{\ell}\intl(u_0,u_1,\cdots,u_{n-1})\bigg)\\
&=&\int\intl(\int^{\widehat{\ell}+1}u_{n-\overline{\ell}},\int^{\widehat{m}+1} u_{n-\overline{\ell}+1},\cdots,\int^{\widehat{\ell}+1} u_{n-1},\int^{\widehat{\ell}} u_0,\cdots,\int^{\widehat{\ell}}u_{n-\overline{\ell}-1})\\
&&\quad\quad\quad\quad\quad\quad(\text{by the induction hypothesis})\\
&=&\intl(\int^{\widehat{\ell}+1}u_{n-\overline{\ell}-1},\int^{\widehat{\ell}+1}u_{n-\overline{\ell}},\int^{\widehat{\ell}+1} u_{n-\overline{\ell}+1},\cdots,\int^{\widehat{\ell}+1} u_{n-1},\int^{\widehat{\ell}} u_0,\cdots,\int^{\widehat{\ell}}u_{n-\overline{\ell}-2}).\\
&&\quad\quad\quad\quad\quad\quad(\text{by Lemma~\mref{lem:intlin}~(\mref{it:intlac})})
\end{eqnarray*}
Two cases arise.
\smallskip

\noindent
{\bf Case 1.  $\overline{\ell}=n-1$}. Then $\overline{\ell+1}=0$ and $\ell+1=(\hat{\ell}+1)n$ by the equation $\ell=\hat{\ell}n+\overline{\ell}$, and so $\widehat{\ell+1}=\hat{\ell}+1$. Thus, we have
\begin{eqnarray*}
&&\int^{\ell +1}\intl(u_0,u_1,\cdots,u_{n-1})\\
&=&\intl(\int^{\widehat{\ell}+1}u_{n-\overline{\ell}-1},\int^{\widehat{\ell}+1}u_{n-\overline{\ell}},\int^{\widehat{\ell}+1} u_{n-\overline{\ell}+1},\cdots,\int^{\widehat{\ell}+1} u_{n-1},\int^{\widehat{\ell}} u_0,\cdots,\int^{\widehat{\ell}}u_{n-\overline{\ell}-2})\\
&=&\intl(\int^{\widehat{\ell}+1}u_{0},\int^{\widehat{\ell}+1}u_{1},\int^{\widehat{\ell}+1} u_{n-\overline{\ell}+1},\cdots,\int^{\widehat{\ell}+1} u_{n-1})\\
&=&\intl(\int^{\widehat{\ell+1}}u_{0},\int^{\widehat{\ell+1}}u_{1},\int^{\widehat{\ell+1}} u_2,\cdots,\int^{\widehat{\ell+1}} u_{n-1}).
\end{eqnarray*}
\smallskip

\noindent
{\bf Case 2.  $0\leq\overline{\ell}\leq n-2$}. Then $1\leq \overline{\ell+1}=\overline{\ell}+1\leq n-1$ and  $\widehat{\ell+1}=\hat{\ell}$. Hence, we get
\begin{eqnarray*}
&&\int^{\ell +1}\intl(u_0,u_1,\cdots,u_{n-1})\\
&=&\intl(\int^{\widehat{\ell}+1}u_{n-\overline{\ell}-1},\int^{\widehat{\ell}+1}u_{n-\overline{\ell}},\int^{\widehat{\ell}+1} u_{n-\overline{\ell}+1},\cdots,\int^{\widehat{\ell}+1} u_{n-1},\int^{\widehat{\ell}} u_0,\cdots,\int^{\widehat{\ell}}u_{n-\overline{\ell}-2})\\
&=&\intl(\int^{\widehat{\ell+1}+1}u_{n-\overline{\ell}-1},\int^{\widehat{\ell+1}+1}u_{n-\overline{\ell}},\int^{\widehat{\ell+1}+1} u_{n-\overline{\ell}+1},\cdots,\int^{\widehat{\ell+1}+1} u_{n-1},\int^{\widehat{\ell+1}} u_0,\cdots,\int^{\widehat{\ell+1}}u_{n-\overline{\ell}-2})\\
&=&\intl(\int^{\widehat{\ell+1}+1}u_{n-\overline{\ell+1}},\int^{\widehat{\ell+1}+1} u_{n-\overline{\ell+1}+1},\cdots,\int^{\widehat{\ell+1}+1} u_{n-1},\int^{\widehat{\ell+1}} u_0,\cdots,\int^{\widehat{\ell+1}}u_{n-\overline{\ell+1}-1}).
\end{eqnarray*}
This completes the induction, and thus the proof of Eq.~(\mref{eq:intmfold}).
\end{proof}

\subsection{Unlacing of Hurwitz series}
\mlabel{subsec:Hadpro}

According to Items~(\mref{it:addhom}) and ~(\mref{it:klin}) of Lemma~\mref{lem:intlin},  $\intl:HA^n\to HA$ is a $A$-linear map. We next define an inverse of the map $\intl$.
\begin{defn}
Let $n\in\NN^+$.  Let $f\in HA$. A map from $HA$ to $ HA^n$
defined by
\begin{equation}
f\mapsto (f_0,f_1,\cdots,f_{n-1}),
\end{equation}
where for $0\leq i\leq n-1$,  $f_i$ is defined by
\begin{equation}
f_i(k):=f(kn+i)\quad \text{for all}\,\, k\geq 0,
\mlabel{eq:unlacing}
\end{equation}
is called the {\bf unlacing} of $f$ and is denoted by $\unl:HA\to HA^n$.
\mlabel{defn:unlacing}
\end{defn}

\begin{exam} Let $f=(0,1,\cdots,k,\cdots)\in HA$. Let $n=3$. Then
$\unl(f)=(f_0,f_1,f_2)\in HA^3$, where $f_0:=(0,3,6,\cdots,3k,\cdots)$, $f_1:=(1,4,7,\cdots,3k+1,\cdots)$ and $f_2:=(2,5,8,\cdots,3k+2,\cdots)$.
\end{exam}
Now we can see that the map $\unl:HA\to HA^n$ is also an $A$-linear map.
\begin{lemma}Let $n\in\NN^+$. Let $f,g\in HA$. Then
\begin{enumerate}
\item
$\unl(f+g)=\unl(f)+\unl(g).$
\mlabel{it:unllin}
\item For every $a\in A$,
$a\,\unl(f)=\unl(af)$.
\mlabel{it:kunl}
\end{enumerate}
\end{lemma}
\begin{proof}
\smallskip

\noindent
(\mref{it:unllin}) Suppose that $\unl(f+g)=(h_0,h_1,\cdots,h_{n-1})$. Then by the definition of unlacing, for $0\leq i\leq n-1$,
$h_i(k)=(f+g)(kn+i)$ for all $k\geq 0$. By the definition of unlacing again,
$\unl(f)+\unl(g)=(f_0+g_0,f_1+g_1,\cdots,f_{n-1}+g_{n-1})$,
where for $0\leq i\leq n-1$,$f_i(k)=f(nk+i)$ and $g_i(k)=g(nk+i)$ for all $k\geq 0$.
So $(f_i+g_i)(k)=f(nk+i)+g(nk+i)=h_i(k)$ for all $k\geq 0$. Thus, $\unl(f+g)=\unl(f)+\unl(g)$.
\smallskip

\noindent
(\mref{it:kunl})
Let $a\in A$. Suppose that $\unl(af)=(h_0,h_1,\cdots,h_{n-1})$. Then $h_i(k)=(af)(nk+i)=af(nk+i)$. Since $a\,\unl(f)=(af_0,af_1,\cdots,af_{n-1})$, $h_i(k)=af_i(k)$ for all $k\geq 0$, and thus $a\,\unl(f)=\unl(af)$.
\end{proof}
The next observation gives the fact that the map unlacing is an inverse of interlacing.
\begin{prop}
Let $n\in\NN^+$. Then $\intl:HA^n\to HA$ is a bijection, and so is $\unl:HA\to HA^n$.
\mlabel{prop:intunl}
\end{prop}
\begin{proof}We shall prove that $\intl\circ \unl=\id_{HA}$ and $\unl \circ\intl =\id_{HA^n}$. For any given $h\in HA$ and every $0\leq i\leq n-1$,
we let $h_i(k):=h(nk+i)$ for all $k\geq 0$. Then $h_0,h_1,\cdots,h_{n-1}\in HA$. By the definitions of unlacing and interlacing,
\begin{eqnarray*}
(\intl\circ\unl)(h)&=&\intl (\unl(h))\\
&=&\intl(h_0,h_1,\cdots,h_{n-1})\\
&=&h.
\end{eqnarray*}
On the other hand, let $f:=(f_0,f_1,\cdots,f_{n-1})\in HA^n$. Then
\begin{eqnarray*}
(\unl\circ\intl)(f)&=&\unl(\intl(f))\\
&=&(g_0,g_1,\cdots,g_{n-1}),
\end{eqnarray*}
where for $0\leq i\leq n-1$,
$$g_i(k):=\intl(f)(nk+i)$$
 for all $k\geq 0$.
 Since
 \begin{eqnarray*}
 \intl(f)(nk+i)&=&f_{\bar{nk+i}}(\widehat{nk+i})\\
 &=&f_i(k+\widehat{i})=f_i(k).
  \end{eqnarray*}
 and thus $g_i=f_i$ for all $0\leq i\leq n-1$. This means that $(\unl\circ\intl)(f)=f$.
\end{proof}
\subsection{Hadamard product}
\mlabel{subsec:Hadpro1}

As in ~\cite{Shar}, we also give the definition of Hadamard products of Hurwitz series as follows.
\begin{defn}\mlabel{defn:had}
Let $f,g\in HA$. The {\bf Hadamard product} of $f$ and $g$, denoted by $f\bx g$, is the Hurwitz series defined by
\begin{equation}
(f\bx g)(k)=f(k)g(k)\quad \text{for all}\,\,k\in\NN.
\mlabel{eq:hada}
\end{equation}
\end{defn}
To avoid confusion,
the usual Hurwitz product is denoted by juxtaposition of Hurwitz series, and we will always use $\bx$
for the Hadamard product.
\begin{defn}
Let $n\in\NN^+$. For every $1\leq \tau\leq n$, define a Hurwitz series $\tau_n\in HA$ by defining
\begin{equation}
\tau_n(k)=kn+\tau\quad \text{for all}\,\, k\in\NN.
\end{equation}
In other words, $\tau_n:=(\tau,n+\tau,2n+\tau,\cdots)$.
\end{defn}
By Definition~\mref{defn:had},  for any $f\in HA$,
\begin{equation}
(\tau_n\bx f)(k)=(kn+\tau)f(k).
\mlabel{eq:tauf}
\end{equation}
We next explore some basics properties of Hadamard products.
Denote $\bar{1}:=(1,1,1,\cdots)$.
\begin{prop}
Let $f,g,h\in HA$. Let $i\in\NN$.
\begin{enumerate}
\item $f\bx g=g\bx f$\,\,$(${\bf commutative laws}$)$.
\mlabel{it:com}
\item $(f\bx g)\bx h=f\bx (g\bx h)$\,\,$(${\bf associative laws}$)$.
\mlabel{it:ass}
\item $(f+g)\bx h=f\bx h+g\bx h$\,\,$(${\bf distributive laws}$)$.
\mlabel{it:dis}
\item For all $a\in A$, $a (f\bx g)=(af)\bx g=f\bx (ag)$.
\mlabel{it:liner}
\item$\bar{1}\bx f=f\bx \bar{1}=f$\,\,$(${\bf Hadamard identity}$)$.
\mlabel{it:iden}
\item
$\partial^i(f\bx g)=\partial^i(f)\bx \partial^i(g)$.
\mlabel{it:parcom}
\item
$\int^i(f\bx g)=\int^i f\bx\int^i g$.
\mlabel{it:intcom}
\end{enumerate}
\end{prop}
\begin{proof}The proof of Items (\mref{it:com}) -- (\mref{it:iden}) is trivial since $A$ is commutative ring with identity $1$.
Items (\mref{it:parcom}\,) and (\mref{it:intcom}) follow from the definitions of $\partial$ and $\int$.
\end{proof}
\begin{prop}
Let $n\in\NN^+$. Let $z_i,u_i\in HA$ for $0\leq i\leq n-1$.  Let $z=\intl(z_0,z_1,\cdots,z_{n-1})$ and let $u=\intl(u_0,u_1,\cdots,u_{n-1})$. Then
\begin{equation}
\intl(z_0\bx u_0,z_1\bx u_1,\cdots,z_{n-1}\bx u_{n-1})=\intl(z_0,z_1,\cdots,z_{n-1})\bx \intl(u_0,u_1,\cdots,u_{n-1}).
\mlabel{eq:inthad}
\end{equation}
\end{prop}
\begin{proof}
By Eq.~(\mref{eq:hada}),
$(z_{\overline{k}}\bx u_{\overline{k}})(\widehat{k})=z_{\overline{k}}(\widehat{k})u_{\overline{k}}(\widehat{k})$ for all $ k\in\NN$. Thus Eq.~(\mref{eq:inthad}) holds.
\end{proof}

\vspace{0.2cm}
\section{Multiplication of interlacing of Hurwitz series by basis elements}
\mlabel{sec:formula}

Let $z_0,z_1,\cdots,z_{n-1}\in HA$. In order to give the multiplication formula for basis elements $\lge m\rge $ and $\intl(z_0,z_1,\cdots,z_{n-1})$, we shall adopt the following conventions. If $k<n$, then ${k\choose n}=0$, and if $j<0$, $z_{i}(j)=0$ for  $0\leq i\leq n-1$.

\begin{lemma}\mlabel{lem:mfn}
For any $m,k\in\NN$ and all $f\in HA$, we have
\begin{equation}
(\lge m\rge f)(k)={k\choose m} f(k-m).
\end{equation}
\end{lemma}
\begin{proof}
By Eq.~(\mref{eq:Hproduct}), we obtain
$$(\lge m\rge f)(k)=\sum_{i=0}^k{k\choose i} \lge m\rge (i)f(k-i).$$
There are now two cases. If $k<m$, then $\lge m\rge (i)=0$ for $0\leq i\leq k$, and so $(\lge m\rge f)(k)=0$. If $k\geq m$, then we obtain $$\lge m\rge(i)=\left\{\begin{array}{ccc}1,&i=m,\\
0,&\text{otherwise}.\end{array}\right.$$ 
Thus,
$$\hspace{5.4cm}(\lge m\rge f)(k)={k\choose m} f(k-m).\hspace{5.4cm}\mbox{\qedhere}$$
\end{proof}
By Lemma~\mref{lem:mfn}, we obtain
$$(\lge m\rge \lge n\rge )(k)={k\choose m}\lge n\rge(k-m)=\left\{\begin{array}{ccc}0,&k\neq n+m;\\
{m+n\choose m},&k=n+m,\end{array}\right.\quad \text{for all}\,\, k\in\NN.$$
Thus, $\lge m\rge \lge n\rge ={m+n\choose m}\lge m+n\rge$.  This also gives Item~(\mref{it:mnpro}) in Proposition~\mref{prop:basismn}.

\begin{theorem}\mlabel{thm:basisz}
Let $n\in\NN^+$ be fixed and let $m\in\NN$. Let $z_0,z_1,\cdots,z_{n-1}\in HA$ and let $z=\intl(z_0,z_1,\cdots,z_{n-1})$. Then
\begin{equation}
(\langle m\rangle \cdot z)(k)={k\choose m} z_{\overline{k-m}}(\widehat{k-m})\quad\text{for all $k\in\NN$}.
\end{equation}
\end{theorem}
\begin{proof}
By Lemma~\mref{lem:mfn}, for any $k\in\NN$,  we have
\begin{align*}
(\lge m\rge\, z)(k)&=\sum_{i=0}^k{k\choose i} \lge m\rge (i) z(k-i)\\
&={k\choose m}z(k-m)\\
\hspace{6cm}&={k\choose m}z_{\overline{k-m}}(\widehat{k-m}).\hspace{6cm}\mbox{\qedhere}
\end{align*}
\end{proof}
\begin{remark}
Let $n\in\NN^+$ be fixed. Let $z_0,z_1,\cdots,z_{n-1}\in HA$ and let $z=\intl(z_0,z_1,\cdots,z_{n-1})$. For all $f\in HA$ and all $k\in\NN$,
\begin{equation}
(f\, z)(k)=\sum_{i=0}^k{k\choose i}f(i)z_{\overline{k-i}}(\widehat{k-i}).
\mlabel{eq:fintl1}
\end{equation}
\end{remark}

By Definition~\mref{defn:had} and Eq.~(\mref{eq:tauf}), for every $i\in\NN_n$, we get
\begin{equation}
((i+1)_{n}\bx z_i)(k)=(kn+i+1)z_i(k)\quad \text{for all}\,\, k\in\NN.
\mlabel{eq:kni1}
\end{equation}
Thus,
\begin{equation}
(i+1)_{n}\bx z_i=\bigg((i+1)z_i(0),(n+i+1)z_i(1),\cdots,(kn+i+1)z_i(k),\cdots\bigg).
\mlabel{eq:notation1}
\end{equation}
For instance, take $i=0$ in the above equation. Then
$$1_n\bx z_0:=\bigg(z_0(0),(n+1)z_0(1),\cdots,(kn+1)z_0(k),\cdots\bigg).$$
Then $\lge 1\rge\,\intl(z_0,z_1,\cdots,z_{n-1})$ can be rewritten as the following formula.
\begin{lemma}\mlabel{lem:1intl}
Let $n\in\NN^+$ be fixed. Let $z=\intl(z_0,z_1,\cdots,z_{n-1})$. Then
\begin{equation}
\lge 1\rge\,z=\intl\bigg(\int( n_n\bx z_{n-1}),1_n\bx z_0,\cdots,(i+1)_n\bx z_i,\cdots,(n-1)_n\bx z_{n-2}\bigg).
\end{equation}
\end{lemma}
\begin{proof}Let $u_0=\int (n_n\bx z_{n-1})$ and $u_{i+1}=(i+1)_n\bx z_i$ for $0\leq i\leq n-2$. Then we shall prove that for all $k\in\NN$,
$$(\lge 1\rge\,z)(k)=\intl\bigg(u_0,u_1,\cdots,u_i,\cdots,u_{n-1}\bigg)(k).$$
By Theorem~\mref{thm:basisz}, we have
$$\lge 1\rge\,\intl(z_0,z_1,\cdots,z_{n-1})(k)= {k\choose 1}z_{\overline{k-1}}(\widehat{k-1})\quad\text{and}\quad
\intl\bigg(u_0,u_1,\cdots,u_{n-1}\bigg)(k)=u_{\overline{k}}(\widehat{k}).$$
Thus we just verify
\begin{equation}
{k\choose 1}z_{\overline{k-1}}(\widehat{k-1})=u_{\overline{k}}(\widehat{k})\quad\text{for all}\,\,k\in\NN.
\mlabel{eq:zku}
\end{equation}
We distinguish two cases.
\smallskip

\noindent
{\bf Case 1. $\overline{k}=0$}. Let $k=nq$ with $q\in\NN$. If $q=0$, then $k=0$, and so ${k\choose 1}z_{\overline{k-1}}(\widehat{k-1})=0=u_0(0).$ If $q\geq 1$, then Eq.~(\mref{eq:zku}) becomes
$${k\choose 1}z_{n-1}(q-1)=u_{0}(q).$$
By the equation $u_0=\int (n_n\bx z_{n-1})$ and Eq.~(\mref{eq:kni1}), we get
\begin{eqnarray*}
u_{0}(q)&=&(n_n\bx z_{n-1})(q-1)\\
&=&((q-1)n+n)z_{n-1}(q-1)\\
&=&(qn)z_{n-1}(q-1)\\
&=&{k\choose 1}z_{n-1}(q-1).
\end{eqnarray*}
Thus, Eq.~(\mref{eq:zku}) holds for $\overline{k}=0$.

\smallskip

\noindent
{\bf Case 2. $\overline{k}\neq0$}. Then $n\geq 2$ and $\overline{k}=1,\cdots,n-1$, and so $\overline{k-1}=\overline{k}-1$. Moreover, $\widehat{k-1}=\widehat{k}$.
By the equations $u_{i+1}=(i+1)_n\bx z_i$ for $0\leq i\leq n-2$, we have
\begin{eqnarray*}
u_{\overline{k}}(\hat{k})&=&(\overline{k}_n\bx z_{\overline{k}-1})(\hat{k})\quad(\text{by Eq.~(\mref{eq:kni1})})\\
&=&(\hat{k}n+\overline{k})z_{\overline{k}-1}(\hat{k})\\
&=&k z_{\overline{k}-1}(\hat{k})\quad(\text{by Eq.~(\mref{eq:mdeco})})\\
&=&k z_{\overline{k-1}}(\widehat{k-1})\\
&=&{k\choose 1}z_{\overline{k-1}}(\widehat{k-1}).
\end{eqnarray*}
This means that Eq.~(\mref{eq:zku}) holds for all $\overline{k}=1,\cdots,n-1$.
\end{proof}
By Lemma~\mref{lem:intlin}~(\mref{it:intlac}), we obtain
\begin{coro}\mlabel{coro:oneintlac}
Let $n\in\NN^+$ be fixed. Let $z=\intl(z_0,z_1,\cdots,z_{n-1})$. Then
\begin{equation}
\lge 1\rge\,z=\int \intl\bigg(1_n\bx z_0,\cdots,(i+1)_n\bx z_i,\cdots,n_n\bx z_{n-1}\bigg).
\end{equation}
\end{coro}

\begin{lemma}\mlabel{lem:gkintm}
Let $h\in HA$. Then
\begin{equation}
\lge k\rge\int^m h=\int^{k+m}\bigg({k+m\choose k}h(0),\cdots,{k+m+i\choose k}h(i),\cdots\bigg)\quad\text{for all $k,m\in\NN$}.
\mlabel{eq:gkintm}
\end{equation}
\end{lemma}
\begin{proof}
By Lemma~\mref{lem:mfn}, for all $p\in\NN$,
\begin{eqnarray*}
(\lge k\rge \int^m h)(p)&=&{p\choose k} (\int^m h)(p-k)\\
&=&{p\choose k}h(p-k-m).\\
\end{eqnarray*}
Then $(\lge k\rge \int^m h)(p)=0$ for $0\leq p\leq k+m-1$, and for all $p\geq k+m$, we have
$$(\lge k\rge \int^m h)(p)={k+m+i\choose k}h(i),$$
where $i:=p-k-m\geq 0$.
Thus,
$$\hspace{2.6cm}\lge k\rge \int^m h=\int^{k+m}\bigg({k+m\choose k}h(0),\cdots,{k+m+i\choose k}h(i),\cdots\bigg).\hspace{2.6cm}\mbox{\qedhere}$$
\end{proof}
Let $n\in\NN^+$ be fixed and let $k,m\in\NN$. For all $\ell\in\NN$, we define a Hurwitz series $C^\ell_{k,n}\in HA$ by defining
\begin{equation}
C^\ell_{k,n}(p)={\ell+pn\choose k}\quad \text{for all}\,\,p\in\NN.
\mlabel{eq:ckn}
\end{equation}
This means that $C^\ell_{k,n}=\bigg({\ell\choose k},{\ell+n\choose k},{\ell+2n\choose k},\cdots\bigg)$.
Let $z_0,z_1,\cdots,z_{n-1}\in HA$.  By Eq.~(\mref{eq:ckn}), for all $0\leq i\leq n-1$ and all $p\geq 0$, we obtain
\begin{equation}
\bigg(C^{i+k+m}_{k,n}\bx z_i\bigg)(p)={i+k+m+pn\choose k}z_i(p).
\mlabel{eq:actionp}
\end{equation}
This gives
$$C^{i+k+m}_{k,n}\bx z_i=\bigg({i+k+m\choose k}z_i(0),\,{i+k+m+n\choose k}z_i(1),\,\cdots,\,{i+k+m+pn\choose k}z_i(p),\,\cdots,\bigg).$$
If $k=0$, then $C^{i+k+m}_{k,n}\bx z_i=z_i$.  As another example, if $k=1$ and $m=1$, then by Eq.~(\mref{eq:notation1})
\begin{equation}
C^{i+2}_{1,n}\bx z_i=\bigg({i+2\choose 1}z_i(0),\,{i+2+n\choose 1}z_i(1),\cdots,\,{i+2+pn\choose 1}z_i(p),\cdots\bigg)=(i+2)_n\bx z_i.
\mlabel{eq:n2}
\end{equation}

\begin{lemma}\mlabel{lem:mint}
Let $n\in\NN^+$ be fixed. Let $z=\intl(z_0,z_1,\cdots,z_{n-1})$. Then
\begin{equation}
\lge k\rge \int z
=\int^{k+1}\intl\bigg(C^{k+1}_{k,n}\bx z_0,\cdots,C^{i+k+1}_{k,n}\bx z_i,\cdots,C^{n+k}_{k,n}\bx z_{n-1}\bigg)\quad\text{for all  $k\in \NN$}.
\mlabel{eq:kintl}
\end{equation}
\end{lemma}
\begin{proof}Let $u_i:=C^{i+k+1}_{k,n}\bx z_i$ for all $0\leq i\leq n-1$. By Lemma~\mref{lem:mfn}, we shall prove that for all $p\in\NN$,
\begin{equation}
{p\choose k}z_{\overline{p-k-1}}(\reallywidehat{p-k-1})=u_{\overline{p-k-1}}(\reallywidehat{p-k-1}).
\mlabel{eq:intlkn}
\end{equation}
Since $u_i={i+k+1\choose k}_n\bx z_i$ for all $0\leq i\leq n-1$, we have
\begin{eqnarray*}
u_{\overline{p-k-1}}(\reallywidehat{p-k-1})&=&\bigg(C^{\overline{p-k-1}+k+1}_{k,n}\bx z_{\overline{p-k-1}}\bigg)(\reallywidehat{p-k-1})\\
&=&{\overline{p-k-1}+k+1+\reallywidehat{(p-k-1)}n\choose k}z_{\overline{p-k-1}}(\reallywidehat{p-k-1})\quad(\text{by Eq.~(\mref{eq:actionp}}))\\
&=&{p\choose k}z_{\overline{p-k-1}}(\reallywidehat{p-k-1}).\quad(\text{by Eq.~(\mref{eq:mdeco}}))\\
\end{eqnarray*}
Thus, Eq.~(\mref{eq:intlkn}) holds for all $p\in\NN$.
\end{proof}
More generally, we have
\begin{theorem}\mlabel{thm:gmint}
Let $n\in\NN^+$ be fixed. Let $z=\intl(z_0,z_1,\cdots,z_{n-1})$. Then
\begin{equation}
\lge k\rge \int^mz
=\int^{k+m}\intl\bigg(C^{k+m}_{k,n}\bx z_0,\cdots,C^{i+k+m}_{k,n}\bx z_i,\cdots,C^{n-1+k+m}_{k,n}\bx z_{n-1}\bigg)\quad\text{for all}\, k,m\in \NN.
\mlabel{eq:gkintl}
\end{equation}
\end{theorem}
\begin{proof}Let $u_i:=C^{i+k+m}_{k,n}\bx z_i$ for all $0\leq i\leq n-1$. By Lemma~\mref{lem:mfn}, we shall prove that for all $p\in\NN$,
\begin{equation}
{p\choose k}z_{\overline{p-k-m}}(\reallywidehat{p-k-m})=u_{\overline{p-k-m}}(\reallywidehat{p-k-m}).
\mlabel{eq:gintlkn}
\end{equation}
Since $u_i=C^{i+k+m}_{k,n}\bx z_i$ for all $0\leq i\leq n-1$, we have
\begin{eqnarray*}
u_{\overline{p-k-m}}(\reallywidehat{p-k-m})&=&\bigg(C^{\overline{p-k-m}+k+m}_{k,n}\bx z_{\overline{p-k-m}}\bigg)(\reallywidehat{p-k-m})\\
&=&{\overline{p-k-m}+k+m+\reallywidehat{(p-k-m)}n\choose k}z_{\overline{p-k-m}}(\reallywidehat{p-k-m})\quad(\text{by Eq.~(\mref{eq:actionp}}))\\
&=&{p\choose k}z_{\overline{p-k-m}}(\reallywidehat{p-k-m}).\quad(\text{by Eq.~(\mref{eq:mdeco}}))\\
\end{eqnarray*}
Thus, Eq.~(\mref{eq:gintlkn}) holds for all $p\in\NN$.
\end{proof}

From Theorem~\mref{thm:gmint}, we obtain
\begin{coro}\mlabel{coro:nintlac}
Let $n\in\NN^+$ be fixed. Let $z=\intl(z_0,z_1,\cdots,z_{n-1})$. Then
\begin{equation}
\lge k\rge\,\intl(z_0,z_1,\cdots,z_{n-1})
=\int^k\intl\bigg(C^k_{k,n}\bx z_0,\cdots,C^{i+k}_{k,n}\bx z_i,\cdots,C^{n-1+k}_{k,n}\bx z_{n-1}\bigg)\,\,\,\text{for all $k\in\NN$}.
\end{equation}
\end{coro}
\begin{proof}Let $u_i:=C^{i+k}_{k,n}\bx z_i$ for all $0\leq i\leq n-1$.
By Lemma~\mref{lem:mfn}, we shall prove that for all $p\geq 0$,
\begin{equation}
{p\choose k}z_{\overline{p-k}}(\widehat{p-k})=u_{\overline{p-k}}(\widehat{p-k}).
\end{equation}
Since $u_i=C^{i+k}_{k,n}\bx z_i$ for all $0\leq i\leq n-1$, we have
\begin{eqnarray*}
u_{\overline{p-k}}(\widehat{p-k})&=&\bigg(C^{\overline{p-k}+k}_{k,n}\bx z_{\overline{p-k}}\bigg)(\widehat{p-k})\\
&=&{\overline{p-k}+k+(\widehat{p-k})n\choose k}z_{\overline{p-k}}(\widehat{p-k})\quad(\text{by Eq.~(\mref{eq:actionp}}))\\
\hspace{4.7cm}&=&{p\choose k}z_{\overline{p-k}}(\widehat{p-k})\quad(\text{by Eq.~(\mref{eq:mdeco}})).\hspace{4.7cm}\mbox{\qedhere}
\end{eqnarray*}
\end{proof}

We now state the first main result on the multiplication formula for basis elements $\lge k\rge $ and the interlacing of $z_0,z_1,\cdots,z_{n-1}\in HA$.
\begin{theorem}\mlabel{thm:main1}
Let $n\in\NN^+$ be fixed. Let $z=\intl(z_0,z_1,\cdots,z_{n-1})$. For all $k\in\NN$, we have
\begin{eqnarray}
&&\lge k\rge\,\intl(z_0,z_1,\cdots,z_{n-1})\mlabel{eq:inmb}\\
&=&\intl(\int^{\widehat{k}+1}C^{n-\overline{k}+k}_{k,n}\bx z_{n-\overline{k}},\cdots,\int^{\widehat{k}+1} C^{n-1+k}_{k,n}\bx z_{n-1},\int^{\widehat{k}} C^k_{k,n}\bx z_0,\cdots,\int^{\widehat{k}}C^{n-\overline{k}-1+k}_{k,n}\bx z_{n-\overline{k}-1}).\nonumber
\end{eqnarray}
\end{theorem}
\begin{proof}
This follows from Corollary~\mref{coro:nintlac} and Proposition~\mref{prop:intmfold}.
\end{proof}

We finally generalize Theorem~\mref{thm:main1} to the multiplication formula for  arbitrary Hurwitz series $f\in HA $ and $\intl(z_0,z_1,\cdots,z_{n-1})$ by linearity. What is more, this precisely answers the aforementioned  question.
\begin{theorem}
Let $n\in\NN^+$. Let $f\in HA$. Let $z_0,z_1,\cdots,z_{n-1}\in HA$. For any $k\in\NN$, we have
\begin{equation}
f\,\intl(z_0,z_1,\cdots,z_{n-1})=\intl(h_0,h_1,\cdots,h_{n-1}),
\mlabel{eq:fintl}
\end{equation}
where for $0\leq i\leq \overline{k}-1$,
$$h_i=\sum_{k\in\NN}\int^{\hat{k}+1}f(k)C^{n-\overline{k}+k+i}_{k,n}\bx z_{n-\overline{k}+i}
\quad \text{and for}\,\, \overline{k}\leq i\leq n-1,
\quad h_i=\sum_{k\in\NN}\int^{\widehat{k}}f(k) C^{i-\overline{k}+k}_{k,n}\bx z_{i-\overline{k}}.$$
\end{theorem}
\begin{proof}By Lemma~\mref{lem:intlin}~(\mref{it:addhom}) and (\mref{it:klin}), we get
\begin{eqnarray*}
&&f\,\intl(z_0,z_1,\cdots,z_{n-1})\\
&=&\bigg(\sum_{k\in\NN}f(k)\lge k\rge\bigg)\intl(z_0,z_1,\cdots,z_{n-1})\\
&=&\sum_{k\in\NN}\bigg(f(k)\lge k\rge \intl(z_0,z_1,\cdots,z_{n-1})\bigg)\\
&=&\sum_{k\in\NN}\bigg(f(k)\bigg(\lge k\rge \intl(z_0,z_1,\cdots,z_{n-1})\bigg)\bigg)\\
&=&\sum_{k\in\NN}f(k)\intl\bigg(\int^{\widehat{k}+1}C^{n-\overline{k}+k}_{k,n}\bx z_{n-\overline{k}},\cdots,\int^{\widehat{k}+1} C^{n-1+k}_{k,n}\bx z_{n-1},\\
&&\int^{\widehat{k}}C^k_{k,n}\bx z_0,\cdots,\int^{\widehat{k}}C^{n-\overline{k}-1+k}_{k,n}\bx z_{n-\overline{k}-1}\bigg)\\
&=&\intl\bigg(\sum_{k\in\NN}f(k)\int^{\widehat{k}+1}C^{n-\overline{k}+k}_{k,n}\bx z_{n-\overline{k}},\cdots,\sum_{k\in\NN}f(k)\int^{\widehat{k}+1} C^{n-1+k}_{k,n}\bx z_{n-1},\\
&&\sum_{k\in\NN}f(k)\int^{\widehat{k}}C^k_{k,n}\bx z_0,\cdots,\sum_{k\in\NN}f(k)\int^{\widehat{k}}C^{n-\overline{k}-1+k}_{k,n}\bx z_{n-\overline{k}-1}\bigg)\\
\hspace{4cm}&=&\intl(h_0,h_1,\cdots,h_{n-1}).\hspace{7cm}\mbox{\qedhere}
\end{eqnarray*}
\end{proof}

\section{Differential equations}
\mlabel{sec:diffeq}

For now let $A$ be a field of any characteristic.  In this section, we focus on the differential equations in the differential ring of Hurwitz series.
We begin by recalling some notations, as in~\cite{GK1}.
Let $h_0,h_1,\cdots,h_{n-1}\in HA$. Define an \emph{$A$-linear homogeneous differential operator} $L$ on $HA$ as follows:
\begin{equation}
L:HA\to HA,\,y\mapsto \partial^n(y)+\sum_{i=0}^{n-1}h_i\partial^{i}(y).
\end{equation}
Let
\begin{equation}
L(y)=\partial^n(y)+\sum_{i=0}^{n-1}h_i\partial^{i}(y)=0.
\mlabel{eq:homgde}
\end{equation}
 Then we call it
an \emph{$n$th-order monic linear homogeneous differential equation} in $HA$.

\subsection{Natural exponential functions in Hurwitz series ring}
\begin{prop}$($~\cite[Proposition~2.1]{KS}$)$ Let $h_0,h_1,\cdots,h_{n-1}\in HA$ and let $L$ be defined as above.  Let
$$V=\{h\in HA\,|\,L(h)=0\}.$$
Then $V$ is an $n$-dimensional vector space over $A$.
\mlabel{prop:vsp}
\end{prop}

By Proposition~\mref{prop:vsp}, we obtain
\begin{coro}\mlabel{coro:solu}
Let $y_1,\cdots,y_n\in HA$ be linearly independent solutions of $L(y)=0$. Then $y_0$ is a solution of $L(y)=0$ if and only if there exist $c_1,\cdots,c_n\in A$ such that $y_0=\sum_{i=1}^{n}c_iy_i$.
\end{coro}

For all $h:=(h(0),h(1),h(2),\cdots)\in HA$, we define the {\bf order} of $h$, denoted by $\ord(h)$, to be the minimum natural number $i\in\NN$ such that $h(i)\neq 0$ and $h(j)=0 $ for all $j=0,\cdots,i-1$. We adopt the convention that $\ord(h)=\infty$ if $h=0$.  For instance, $\ord(\lge n\rge)=n$ for all $n\in\NN$.

Denote
$$H_0A:=\{h\in HA\,|\, \ord(h)\geq1\}.$$
Also, we set $0\in H_0A$.  Note that $\vep: HA\to A,\,h\mapsto h(0)$, is a ring homomorphism. Then $h\in H_0A$ if and only if $\vep(h)=0$. Let
$$H_1A:=\{h\in HA\,|\,\vep(h)=1\}.$$
An element $h\in HA$ is invertible if and only if $\vep(h) (=h(0))$  is invertible in $A$~\cite[Lemma~3.2]{K2}.  Suppose that $h\in HA$ is invertible.
Then the inverse element of $h$ in $HA$, denoted by $h^{-1}$, is given by the following recursive formula:
\begin{equation}
h^{-1}(0)=h(0)^{-1}\quad\text{and}\quad h^{-1}(n)=-h(0)^{-1}\sum_{k=1}^{n}{n\choose k}h(k)h^{-1}(n-k)\quad\text{for all}\,n\in\NN^+.
\mlabel{eq:inverse}
\end{equation}

We also denote by $H_\ast A$ the group of invertible elements in $HA$. Thus, $H_1A$ is a  normal subgroup of $H_\ast A$ and $H_0A\cap H_\ast A=\emptyset$. Let
$$H_\diamond  A:=\{h\in HA\,|\, 0\neq \vep(h)\,\, \text{is not invertible in}\,A\}.$$
Then we obtain a disjoint union
\begin{equation}
HA=H_0A\bigsqcup H_\ast A\bigsqcup H_\diamond A.
\end{equation}
\begin{defn}\mlabel{defn:dividpow}
Let $n\in\NN$. For all $h\in HA$, we define the $n$th {\bf divided power} of $h$, denoted by $h^{[n]}$, by defining inductively
\begin{equation}
h^{[0]}= 1\quad \text{and}\quad h^{[n]}=\int (h^{[ n-1]}\partial(h)),\,\,\text{for all}\,\,n\in\NN^+.
\end{equation}
\end{defn}
From the definition of $n$th divided power of $h$, we obtain
$$\partial (h^{[n]})=h^{[n-1]}\partial(h),$$
analogous to the power rule in calculus.
Let $x:=\lge 1\rge$. Then
$$x^{[0]}=1=\lge 0\rge, \,x^{[1]}=x=\lge 1\rge, \,x^{[2]}=\lge 2\rge.$$
In general,
$$x^{[n]}=\lge n\rge\,\,\text{for all} \,\,n\in\NN.$$
\begin{prop}\mlabel{prop:nfacpro}
Let $f,g\in HA$ and let $h\in H_0A$. For all $n\in\NN$,
\begin{enumerate}
\item
$n!h^{[n]}=h^n$.
\mlabel{it:nfac}
\item
$(f+g)^{[n]}=\sum_{i+j=n}f^{[i]}g^{[j]}$.
\item
$(fh)^{[n]}=f^nh^{[n]}$.
\item
$f^{[m]}f^{[n]}={m+n \choose m}f^{[m+n]}$.
\item If $h\neq 0$, then $\ord(\pal(h))=\ord(h)-1$.
\mlabel{it:ordp}
\item $\ord(h^{[n]})\geq n\,\ord(h)$, and hence, $\ord(h^{[n]})\geq n$.
\mlabel{it:ordpn}
\item
$h^{[n]}(i)=0$ for all $i<n$.
\mlabel{it:ordpns}
\item
$\ord(fg)\geq \ord(f)+\ord(g).$
\mlabel{it:fgc}
\item
If $f\neq 0$, then $\ord(\int f)=\ord(f)+1$.
 \mlabel{it:ordi}
 \item $\ord(f^{[n]})\geq n$.
\mlabel{it:ordpnf}
\end{enumerate}
\end{prop}
\begin{proof}
The first five Items~(\mref{it:nfac})- (\mref{it:ordpn}) follow from ~\cite[Proposition~2.3]{GK1} and ~\cite[Lemma~2.2]{KP}. The Item~(\mref{it:ordpns}) follows from Item~(\mref{it:ordpn}) and Item~(\mref{it:fgc}) follows from the definition of order of Hurwitz series. We now prove  Item~(\mref{it:ordi}). Assume that $f\neq 0$ and $\ord(f)=k$ for $k\geq 0$. Then $\int f=(\underbrace{0,\cdots,0}_{k+1},f(k),\cdots)$ with $f(k)\neq 0$. Hence,
$$\ord(\int f)=k+1=\ord(f)+1.$$

Finally, we prove Item~(\mref{it:ordpnf}) by induction on $n\geq 0$.  For $n=0$, we get $f^{[n]}=1$, and so $\ord(f^{[0]})=0$.  Assume that $\ord(f^{[k]})\geq k$ has been proved for $k\geq 1$.  If $f^{[k]}\pal (f)=0$, $\ord(f^{[k+1]})= \infty$, and so Item~(\mref{it:ordpnf}) holds. Now we suppose that $f^{[k]}\pal (f)\neq 0$. Since $f^{[k+1]}=\int (f^{[k]}\pal (f))$, we have $\ord(f^{[k+1]})=\ord(f^{[k]}\pal(f))+1$ by Item~(\mref{it:ordi}). By the induction hypothesis, we know that
\delete{
$$f^{[k]}(i)=0\quad\text{for all}\,\, 0\leq i\leq k-1.$$
Then
\begin{equation}
(f^{[k]}\pal(f))(i)=\sum_{j=0}^i{i\choose j}f^{[k]}(j)\pal(f)(i-j)=0\quad \text{for all}\,\, 0\leq i\leq k-1.
\end{equation}
Hence $\ord(f^{[k]}\pal(f))\geq k$, and so,}
\begin{eqnarray*}
\ord(f^{[k+1]})&=&\ord(f^{[k]}\pal(f))+1\\
&\geq&\ord(f^{[k]})+\ord(\pal(f))+1\quad(\text{by Item~(\mref{it:fgc})})\\
&\geq &k+1.
\end{eqnarray*}
This completes the induction.
\end{proof}
According to the $n$th divided power, we define the natural exponential function
\begin{equation}
\exp:H_0A\to H_1A,\,\,h\mapsto \exp(h):=\sum_{n\in\NN}h^{[n]}.
\end{equation}
For example,
\begin{equation*}
\exp(x)=\sum_{n\in\NN}x^{[n]}=\sum_{n\in\NN}\lge n\rge=(1,1,1\cdots)=:\bar{1}.
\end{equation*}
By Proposition~\mref{prop:nfacpro}~(\mref{it:ordpns}), we get
\begin{equation}
\exp(h)(n)=\sum_{k=0}^nh^{[k]} (n)+\sum_{k=n+1}^\infty h^{[k]}(n)=\sum_{k=0}^nh^{[k]} (n)\quad\text{for all}\, n\in\NN.
\mlabel{eq:expn}
\end{equation}
Since $\bar{1}\in H_1A$, we have the inverse of $\bar{1}$, and $\bar{1}^{-1}=(1,-1,1,-1,\cdots)$ by Eq.~(\mref{eq:inverse}). Then  $\sin$ and $\cos$ functions in $HA$ can be defined as follows:
\begin{equation}
\sin:=\intl(0,\bar{1}^{-1})\quad\text{and}\quad \cos:=\intl(\bar{1}^{-1},0).
\mlabel{eq:sincos}
\end{equation}
Note that $\sin\in H_0A$. Then
\begin{equation}
\exp(\sin)=\sum_{n\in\NN}\sin^{[n]}=(1,1,1,0,-3,-8,-3,56,217,64,\cdots).
\mlabel{eq:expsinv}
\end{equation}
But $\exp(\cos)$ does not make sense because $\cos\not\in H_0A$. In addition, since $\exp(\sin)\in H_1A$,  we have
\begin{equation}
\exp(\sin)^{-1}=(1,-1,1,0,-3,8,-3,-56,217,-64,\cdots).
\end{equation}
\begin{prop}\mlabel{prop:derexp}
Let $f\in HA$ and let $g,h\in H_0A$. Then
\begin{enumerate}
\item
$\pal(f^k)=k f^{k-1}\pal(f)$ for all $k\in\ZZ$.
\item
$\exp(h)^{-1}=\exp(-h)$.
\item
$\pal(\exp(h))=\exp(h)\pal(h).$
\end{enumerate}
\end{prop}
\begin{proof}
The first Item follows from that $\pal$ is a derivation and the last two Items follow from ~\cite[Proposition~2.9]{GK1}.
\end{proof}

\begin{exam}\mlabel{exam:sinexp}
Consider the differential equation in $HA$
\begin{equation}
\partial^2 (y)=\cos\cdot \partial(y)-\sin\cdot y.
\mlabel{eq:pcs}
\end{equation}
Firstly, $0$ is a solution of Eq.~(\mref{eq:pcs}). We know that $h\in HA$ is a solution of Eq.~(\mref{eq:pcs}) if and only if
$$h(n+2)=\sum_{i=0}^{n}{n\choose i}\bigg(\cos(i)h(n-i+1)-\sin(i)h(n-i)\bigg)\quad\text{for all}\,n\in\NN$$
or
$$h(n)=\sum_{i=0}^{n-2}{n-2\choose i}\bigg(\cos(i)h(n-i-1)-\sin(i)h(n-i-2)\bigg)\quad\text{for all}\,n\geq2.$$
For each $n=2,3,4,5,6,7,8,9$, we get
\begin{align*}
h(2)&=h(1),&h(3)&=h(2)-h(0),\\
h(4)&=h(3)-3h(1),&h(5)&=h(4)-6h(2)+h(0),\\
h(6)&=h(5)-10h(3)+5h(1),&h(7)&=h(6)-15h(4)+15h(2)-h(0),\\
h(8)&=h(7)-21h(5)+35h(3)-7h(1),&h(9)&=h(8)-28h(6)+70h(4)-28h(2)+h(0).
\end{align*}
Thus,
if we take $h(0)=1$ and $h(1)=0$, we get a solution of Eq.~(\mref{eq:pcs})
$$z_1:=(1,0,0,-1,-1,0,10,24,-11,-360,\cdots),$$
which is a new sequence, that is, a sequence may not be found in OEIS~\cite{Slo}.
In addition,
if we let $h(0)=0$ and $h(1)=1$, then we obtain another solution of Eq.~(\mref{eq:pcs})
\begin{equation}
z_2:=(0,1,1,1,-2,-8,-13,32,228,424,\cdots),
\mlabel{eq:z2}
\end{equation}
which is also a new sequence.  Note that both solutions $z_1$ and $z_2$ are $A$-linearly independent.

By Corollary~\mref{coro:solu}, we know that $c_1z_1+c_2z_2$ are solutions of Eq.~(\mref{eq:pcs}) for all $c_1,c_2\in A$.
In particular, if $c_0=c_1=1$, then
$$z_0:=z_1+z_2=(1,1,1,0,-3,-8,-3,56,217,64,\cdots),$$
is also a solution of  Eq.~(\mref{eq:pcs}), and
we see that $z_0=\exp(\sin)$ by Eq.~(\mref{eq:expsinv}). Furthermore, $z_0$ is exactly the sequence A002017 whose exponential generating function is just $\exp(\sin(x))$  in OEIS.
\end{exam}

\subsection{The general solutions to second-order linear differential equations}
\mlabel{subsec:SolSec}
We next consider a general second-order linear homogeneous differential equation
\begin{equation}
\partial^2(y)+h_1\partial(y)+h_0y=0,
\mlabel{eq:sec-ord}
\end{equation}
where $h_0,h_1\in HA$.

\begin{prop}\mlabel{prop:secsol}
Let $y_1\in H_\ast A$ be an invertible solution of Eq.~(\mref{eq:sec-ord}). Then
\begin{equation}
y_2:=y_1\int \big(\exp(-\int h_1) y_1^{-2}\big)
\mlabel{eq:secsol}
\end{equation}
is also a solution of  Eq.~(\mref{eq:sec-ord}).
\end{prop}
\begin{proof}
We need only verify that
$\pal^2(y_2)+h_1\pal(y_2)+h_0y_2=0.$
Since $y_1$ is a solution of  Eq.~(\mref{eq:sec-ord}), we have
\begin{eqnarray*}
&&\pal^2(y_2)+h_1\pal(y_2)+h_0y_2\\
&=&\pal^2(y_1)\int (\exp(-\int h_1) y_1^{-2})+2\pal(y_1)(\exp(-\int h_1) y_1^{-2})\\
&&+y_1\bigg(\exp(-\int h_1)(-h_1)y_1^{-2}-2\exp(-\int h_1)y_1^{-3}\pal(y_1)\bigg)\quad(\text{by Proposition~\mref{prop:derexp}})\\
&&+h_1\bigg(\pal(y_1)\int \big( \exp(-\int h_1)y_1^{-2}\big)+y_1\exp(-\int h_1)y_1^{-2}\bigg)+h_0y_1\int\big( \exp(-\int h_1)y^{-2}\big)\\
&=&\bigg(\pal^2(y_1)+h_1\pal(y_1)+h_0y_1\bigg)\int \big( \exp(-\int h_1)y_1^{-2}\big)\\
&&+2\pal(y_1)(\exp(-\int h_1) y_1^{-2})-2\exp(-\int h_1)y_1^{-2}\pal(y_1)\\
&&+y_1\exp(-\int h_1)(-h_1)y_1^{-2}+h_1y_1\exp(-\int h_1)y_1^{-2}\\
\hspace{1cm}&=&0.\hspace{13cm}\mbox{\qedhere}
\end{eqnarray*}
\end{proof}
\begin{theorem}\mlabel{thm:mainres}
Let $y_1\in H_\ast A$ be an invertible solution of Eq.~(\mref{eq:sec-ord}).
Let $y_2$ be defined as in Eq.~(\mref{eq:secsol}).
Then $y_1$ and $y_2$ are linearly independent, and further,  the general solution of Eq.~(\mref{eq:sec-ord}) is given by
\begin{equation}
y=y_1 \Big(c_1+c_2\int \big(\exp(-\int h_1) y_1^{-2}\big)\Big)\quad \text{for all}\,\, c_1,c_2\in A.
\end{equation}
\end{theorem}
\begin{proof}
By Proposition~\mref{prop:vsp}, it suffices to prove that $y_1$ and $y_2=y_1\int \big(\exp(-\int h_1) y_1^{-2}\big)$ are linearly independent. Suppose that there exist $a_1,a_2\in A$ such that $a_1y_1+a_2y_2=0$. This gives
$$a_1y_1(0)+a_2y_2(0)=0.$$
 Since
 $$y_2(0)=y_1(0)\int \big(\exp(-\int h_1) y_1^{-2}\big)(0)=0,$$ we have
$a_1y(0)=0$, and so $a_1=0$ by $y(0)$ being invertible. Note that
\begin{eqnarray*}
y_2(1)&=&y_1(1)\int \big(\exp(-\int h_1) y_1^{-2}\big)(0)+y_1(0)\int \big(\exp(-\int h_1) y_1^{-2}\big)(1)\\
&=&y_1(0)\big(\exp(-\int h_1) y_1^{-2}\big)(0)\\
&=&y_1(0)\exp(-\int h_1)(0)\, y_1^{-2}(0)\\
&=&y_1(0)^{-1}.
\end{eqnarray*}
Since $a_2y_2=0$, we have $a_2y_2(1)=0$, and hence $a_2=0$.
\end{proof}

By Example~\mref{exam:sinexp}, we know that $y_1=\exp(\sin)$ is an invertible solution of Eq.~(\mref{eq:pcs})
$$\partial^2 (y)=\cos\cdot \partial(y)-\sin\cdot y.$$
According to Proposition~\mref{prop:secsol}, we obtain another solution
\begin{eqnarray*}
y_2&=&y_1\int \big(\exp(-\int h_1) y_1^{-2}\big)\\
&=&\exp(\sin)\int(\exp(\int \cos)\exp(-\sin)^2)\quad(\text{by}\,\exp(\sin)^{-1}=\exp(-\sin) )\\
&=&\exp(\sin)\int(\exp(\sin)\exp(-\sin)^2\quad(\text{by}\,\int \cos =\sin)\\
&=&\exp(\sin)\int\exp(-\sin).
\end{eqnarray*}
In fact, $y_2$ is exactly $z_2=(0,1,1,1,-2,-8,-13,32,228,424,\cdots)$ defined as in Eq.~(\mref{eq:z2}).
Thus, by Theorem~\mref{thm:mainres}, we get the general solution of Eq.~(\mref{eq:pcs}):
$$y=\exp(\sin)\Big(c_1+c_2\int\exp(-\sin)\Big).$$

Take some more examples on the second-order differential equations.
\begin{exam} \mlabel{exam:ordx}
Consider a second-order differential equation
\begin{equation}
\pal^2(y)+(1+x)\pal(y)-y=0.
\mlabel{eq:ordx}
\end{equation}
We know that $y=1+x\in H_1A$  is an invertible solution of Eq.~(\mref{eq:ordx}). Then the general solution of Eq.~(\mref{eq:ordx}) is
\begin{equation}
y=(1+x)\bigg(c_1+c_2\int \Big(\exp(-\int (1+x)) (1+x)^{-2}\Big)\bigg).
\mlabel{eq:gener2}
\end{equation}
\end{exam}
\delete{
\begin{problem}
Let $h_0,h_1\in HA$. Let
\begin{equation}
\pal^2(y)+h_1\pal(y)+h_0y=0.
\mlabel{eq:order2}
\end{equation}
If $y_0,y_1$ is a basis for the solution space of Eq.~(\mref{eq:order2}), then there exists a unique invertible solution of Eq.~(\mref{eq:order2}).
\end{problem}
}
\begin{exam}
By Example~\mref{exam:ordx}, we see that $y_0:=1+x$ is an invertible solution of Eq.~(\mref{eq:ordx}). Furthermore, if we take $c_1=c_2=1$ in Eq.~(\mref{eq:gener2}), we obtain another solution
$$y_1:=(1+x)\bigg(1+\int \Big(\exp(-\int (1+x)) (1+x)^{-2}\Big)\bigg).$$
Since $y_0$ and $y_1$ are linearly independent,  $y_0,y_1$ forms a basis for the solution space of Eq.~(\mref{eq:ordx}).
\end{exam}

\subsection{The method of reduction of order}
\mlabel{sec:RedOrd}

We now discuss the reduction of order in $HA$. Let $y_1\in H_\ast A$ be an invertible solution of Eq.~(\mref{eq:sec-ord}).
Suppose that $y_2=hy_1$ is another solution of Eq.~(\mref{eq:sec-ord}) for some $h\in HA$.
Then
$$\partial^2(y_2)+h_1\pal(y_2)+h_0y_2=0.$$
Since $\pal$ is a derivation and $y_1$ is a solution of Eq.~(\mref{eq:sec-ord}), we have
\begin{eqnarray*}
&&\pal^2(h y_1)+h_1\pal(h y_1)+h_0(h y_1)\\
&=&\pal^2(h)y_1+2\pal(h)\pal(y_1)+h\pal^2(y_1)
+h_1(\pal(h)y_1)+h\pal(y_1))+h_0(h y_1)\\
&=&\pal^2(h)y_1+2\pal(h)\pal(y_1)+h_1\pal(h)y_1.
\end{eqnarray*}
Hence,
$$\pal^2(h)y_1+2\pal(h)\pal(y_1)+h_1\pal(h)y_1=0.$$
Since $y_1\in H_\ast A$ is invertible, we know that
$$\pal^2(h)=-\pal(h)\big(2\pal(y_1)y_1^{-1}+h_1\big).$$
This gives a recursive formula
\begin{equation}
h(n+2)=-\sum_{i=0}^{n}{n\choose i} h(i+1)\big(2\pal(y_1)y_1^{-1}+h_1\big)(n-i)\quad\text{for all}\,\,n\in\NN.
\mlabel{eq:secrecfor}
\end{equation}
In order to simplify certain calculations, we set $\Delta=2\pal(y_1)y_1^{-1}+h_1$. Thus, Eq.~(\mref{eq:secrecfor}) becomes
\begin{equation}
h(n+2)=-\sum_{i=0}^{n}{n\choose i} h(i+1)\Delta(n-i)\quad\text{for all}\,n\in\NN.
\mlabel{eq:secrecfor1}
\end{equation}
For instance, for each $n=0,1,2,3$, we obtain
\begin{align*}
h(2)&=-h(1)\Delta(0)\\
h(3)&=-\bigg(h(1)\Delta(1)+h(2)\Delta(0)\bigg)=-h(1)\bigg(\Delta(1)-\Delta(0)^2\bigg)\\
h(4)&=-\bigg(h(1)\Delta(2)+2h(2)\Delta(1)+h(3)\Delta(0)\bigg)=-h(1)\bigg(\Delta(2)-3\Delta(0)\Delta(1)+\Delta(0)^3\bigg)\\
h(5)&=-\bigg(h(1)\Delta(3)+3h(2)\Delta(2)+3h(3)\Delta(1)+h(4)\Delta(0)\bigg)\\
&=-h(1)\bigg(\Delta(3)-4\Delta(0)\Delta(2)-3\Delta(1)^2+6\Delta(0)^2\Delta(1)-\Delta(0)^4\bigg).
\end{align*}
Thus, if we let $h(0)=a$ and $h(1)=b$ for $a,b\in A$,  then we  obtain
$$h=(a,b,-b\Delta(0),-b(\Delta(1)-\Delta(0)^2),\cdots).$$

According to the method of  reduction of order~\cite{BD} and Abel's Theorem~\cite[Theorem~1.2]{Green}, the  second-order linear homogeneous differential equation in $A[[t]]$:
\begin{equation}
u''(t)+p(t)u'(t)+q(t)u(t)=0
\mlabel{eq:secordpo}
\end{equation}
has a solution
\begin{equation*}
u_2(t)=u_1(t)\int v(t) u_1(t)^{-2}dt,
\end{equation*}
where $u_1(t)$ is any non-zero solution of Eq.~(\mref{eq:secordpo}) and $v(t)=e^{-\int p(t)dt}$.
Since $u_1(t)$ and $u_2(t)$ are linearly independent, we know that the general of solution of Eq.~(\mref{eq:secordpo}) is
\begin{equation}
y=u_1(t)\Big(c_1+c_2\int e^{-\int p(t)dt}u_1(t)^{-2} dt\Big)\quad\text{for all}\,\,c_1,c_2\in A.
\mlabel{eq:genesolu}
\end{equation}

Suppose that $\QQ\subseteq A$. Then we know that there exists an  isomorphism of differential algebras
\begin{equation}
\Phi:HA\to A[[t]],\quad h\mapsto h(t):=\sum_{n=0}^\infty\frac{h(n)}{n!} t^n.
\mlabel{eq:cong}
\end{equation}
Consequently,  the inverse of $\Phi$ is
\begin{equation}
\Phi^{-1}: A[[t]]\to HA,\quad \sum_{n=0}^\infty h(n) t^n\mapsto (n!h(n)).
\mlabel{eq:invcong}
\end{equation}

Furthermore, we have
\begin{theorem}Let $\QQ\subseteq A$. Let $y_1\in H_\ast A$ be an invertible solution of Eq.~(\mref{eq:sec-ord}). Suppose that $y_2=hy_1$ is another solution of Eq.~(\mref{eq:sec-ord}) for some $h\in HA$.  Then
\begin{equation}
h=h(1)y_1(0)^2\int\Big(\exp(-\int h_1)y_1^{-2}\Big)+h(0).
\mlabel{eq:hspe}
\end{equation}
\end{theorem}
\begin{proof}
Under the isomorphism $\Phi$  in Eq.~(\mref{eq:cong}),  Eq.~(\mref{eq:sec-ord})
becomes a second-order linear homogeneous differential equation in $A[[t]]$
\begin{equation}
y''(t)+h_1(t)y'(t)+h_0(t)y(t)=0.
\mlabel{eq:secordha}
\end{equation}
Thus, $y_2(t)=h(t)y_1(t)$ is a solution of Eq.~(\mref{eq:secordha}).
By Eq.~(\mref{eq:genesolu}), we have
$$y_2=h(t)y_1(t)=y_1(t)\Big(c_1+c_2\int e^{-\int h_1(t)dt}y_1(t)^{-2} dt\Big)\quad\text{for all}\,\,c_1,c_2\in A.$$
This shows that
$$h(t)=c_1+c_2\int e^{-\int h_1(t)dt}y_1(t)^{-2} dt.$$
Applying the isomorphism $\Phi^{-1}$ defined in Eq.~(\mref{eq:invcong}) to the above equation, we get
$$h=c_1+c_2\int \Big(\exp(-\int h_1 )y_1^{-2}\Big).$$
Furthermore, it follows that
$$c_1=h(0) \quad\text{and}\quad c_2=h(1)y_1(0)^2,$$
as desired.
\end{proof}

\smallskip
\noindent {\bf Acknowledgements}: This work was supported by the National Natural Science Foundation of China (Grant No.  11601199  and 11961031). The author thanks Professor William F. Keigher  for many useful conversations and suggestions, and thanks Rutgers University-Newark for its hospitality.

\end{document}